\documentclass[11pt,a4paper]{article}

\usepackage{tikz}
\usepackage{amsthm}
\usepackage{amsmath}
\usepackage{amssymb}
\usepackage{mathrsfs}
\usepackage{geometry}
\usepackage{graphicx}
\usepackage{amsfonts}
\usepackage{epstopdf}
\usepackage{enumerate}
\usepackage{hyperref}

\newcommand{\md}{\mathrm{d}}
\newcommand{\loc}{{\mathrm{loc}}}

\newcommand{\R}{\mathbb{R}}

\newcommand{\Rmnum}[1]{\uppercase\expandafter{\romannumeral#1}} 

\newcommand{\Ee}{\mathcal{E}}

\newcommand{\Ff}{\mathcal{F}}

\newcommand{\myset}[1]{\left\{#1\right\}}

\newtheorem{mythm}{Theorem}[section]
\newtheorem{myprop}[mythm]{Proposition}
\newtheorem{mylem}[mythm]{Lemma}
\newtheorem{mycor}[mythm]{Corollary}
\newtheorem{myrmk}[mythm]{Remark}

\newtheorem{mydef}[mythm]{Definition}

\AtEndDocument{\bigskip{\footnotesize%
  \textsc{Fakult\"{a}t f\"{u}r Mathematik, Universit\"{a}t Bielefeld, Postfach 100131, 33501 Bielefeld, Germany.} \par  
  \textit{E-mail address}: \texttt{ymeng@math.uni-bielefeld.de} \par
  and \par
  \textsc{Department of Mathematical Sciences, Tsinghua University, Beijing 100084, China.} \par
  \textit{E-mail address}: \texttt{meng-yang13@mails.tsinghua.edu.cn}
}}

\begin{document}

\title{Equivalent Semi-Norms of Non-Local Dirichlet Forms on the Sierpi\'nski Gasket and Applications}
\author{Meng Yang}
\date{}

\maketitle

\abstract{We construct equivalent semi-norms of non-local Dirichlet forms on the Sierpi\'nski gasket and apply these semi-norms to a convergence problem and a trace problem. We also construct explicitly a sequence of non-local Dirichlet forms with jumping kernels equivalent to $|x-y|^{-\alpha-\beta}$ that converges exactly to local Dirichlet form.}

\footnote{\textsl{Date}: \today}
\footnote{The author was supported by SFB701 of the German Research Council (DFG). The author is very grateful to Professor Alexander Grigor'yan for very helpful discussions.}

\section{Introduction}

Let us recall the following classical result
\begin{equation}\label{eqn_classical}
\lim_{\beta\uparrow2}(2-\beta)\int_{\R^n}\int_{\R^n}\frac{(u(x)-u(y))^2}{|x-y|^{n+\beta}}\md x\md y=C(n)\int_{\R^n}|\nabla u(x)|^2\md x,
\end{equation}
for all $u\in W^{1,2}(\R^n)$, where $C(n)$ is some positive constant (see \cite[Example 1.4.1]{FOT11}). Probabilistically, the subordination process of the Brownian motion can approximate the Brownian motion in some sense with appropriate time change. The purpose of this paper is to prove an analog result for the Sierpi\'nski gasket instead of $\R^n$.

Consider the following points in $\R^2$: $p_0=(0,0)$, $p_1=(1,0)$, $p_2=({1}/{2},{\sqrt{3}}/{2})$. Let $f_i(x)=(x+p_i)/2$, $x\in\R^2$, $i=0,1,2$. Then the Sierpi\'nski gasket (SG) is the unique non-empty compact set $K$ such that $K=f_0(K)\cup f_1(K)\cup f_2(K)$. Let
$$V_0=\myset{p_0,p_1,p_2},V_{n+1}=f_0(V_n)\cup f_1(V_n)\cup f_2(V_n)\text{ for all }n\ge0.$$
Then $\myset{V_n}$ is an increasing sequence of finite sets and $K$ is the closure of $\cup_{n=0}^\infty V_n$. For all $n\ge1$, let
$$W_n=\myset{w=w_1\ldots w_n:w_i=0,1,2,i=1,\ldots,n},$$
for all $w=w_1\ldots w_n\in W_n$, let
$$
\begin{aligned}
V_w&=f_{w_1}\circ\ldots\circ f_{w_n}(V_0),\\
K_w&=f_{w_1}\circ\ldots\circ f_{w_n}(K),\\
P_w&=f_{w_1}\circ\ldots\circ f_{w_{n-1}}(p_{w_n}).
\end{aligned}
$$
Let $\nu$ be the normalized Hausdorff measure on $K$. Let $(\Ee_\beta,\Ff_\beta)$ be given by
$$
\begin{cases}
\Ee_\beta(u,u)=\int_K\int_K\frac{(u(x)-u(y))^2}{|x-y|^{\alpha+\beta}}\nu(\md x)\nu(\md y),\\
\Ff_\beta=\myset{u\in C(K):\Ee_\beta(u,u)<+\infty}.
\end{cases}
$$
It is known that $(\Ee_\beta,\Ff_\beta)$ is a non-local regular Dirichlet form on $L^2(K;\nu)$ for all $\beta\in(\alpha,\beta^*)$, where $\alpha=\log3/\log2$ is the Hausdorff dimension, $\beta^*=\log5/\log2$ is the walk dimension of SG (see \cite{Kum03} using heat kernel estimates and subordination technique, \cite{GY17} using trace theory of Dirichlet form and \cite{KL16} using effective resistance on graph).

Let $(\Ee_{\loc},\Ff_{\loc})$ be given by
$$
\begin{cases}
\Ee_\loc(u,u)=\lim_{n\to+\infty}(\frac{5}{3})^n\sum_{w\in W_n}\sum_{p,q\in V_w}(u(p)-u(q))^2,\\
\Ff_\loc=\myset{u\in C(K):\Ee_\loc(u,u)<+\infty}.
\end{cases}
$$
It is known that $(\Ee_\loc,\Ff_\loc)$ is a local regular Dirichlet form on $L^2(K;\nu)$ which corresponds to the diffusion on SG (see \cite{Bar98,Kig01}).

Analog to (\ref{eqn_classical}), one may expect that $(\beta^*-\beta)\Ee_\beta(u,u)$ converges to $\Ee_\loc(u,u)$. However, this is not known. Using the sub-Gaussian estimates for the heat kernel of $\Ee_\loc$, it was shown in \cite[Theorem 3.1]{Pie08}, \cite[2.1]{Kum03} that the Dirichlet form $\tilde{\Ee}_\beta$ that is obtained from $\Ee_\loc$ by subordination of order $\beta/\beta^*$ has the following properties
\begin{equation}\label{eqn_sub}
\begin{aligned}
&\tilde{\Ee}_\beta(u,u)\asymp(\beta^*-\beta)\Ee_\beta(u,u),\\
&\tilde{\Ee}_\beta(u,u)\to\Ee_\loc(u,u)\text{ as }\beta\uparrow\beta^*.
\end{aligned}
\end{equation}
Moreover, the jump kernel of $\tilde{\Ee}_\beta$ is of the order $|x-y|^{-(\alpha+\beta)}$ for all $\beta\in(0,\beta^*)$.

In the present paper, we construct explicitly a different semi-norm $E_\beta$ of jump type that has properties similar to (\ref{eqn_sub}). Our construction has the following two advantages. First, our construction is independent of any knowledge about the local Dirichlet form $\Ee_\loc$ except for its definition. Second, we obtain a monotone convergence result for all functions in $L^2(K;\nu)$ which implies a Mosco convergence. While \cite[Theorem 3.1]{Pie08} only gave a convergence result for functions in $\Ff_\loc$.

The new semi-norm $E_\beta$ is defined as follows.
$$E_\beta(u,u):=\sum_{n=1}^\infty2^{(\beta-\alpha)n}\sum_{w\in W_n}\sum_{p,q\in V_w}(u(p)-u(q))^2.$$
We state the main results in the next two theorems. Our first main result is as follows.
\begin{mythm}\label{thm_main}
For all $\beta\in(\alpha,+\infty)$, for all $u\in C(K)$, we have
$$E_\beta(u,u)\asymp\Ee_\beta(u,u).$$
\end{mythm}

Recall that a similar result for the unit interval was proved in \cite{Kam97} as follows. Let $I=[0,1]$. Then for all $\beta\in(1,+\infty)$, for all $u\in C(I)$, we have
\begin{equation}\label{eqn_interval}
\sum_{n=1}^\infty2^{(\beta-1)n}\sum_{i=0}^{2^n-1}(u(\frac{i}{2^n})-u(\frac{i+1}{2^n}))^2\asymp\int_{I}\int_I\frac{(u(x)-u(y))^2}{|x-y|^{1+\beta}}\md x\md y.
\end{equation}

Consider the convergence problem. Assume that $(\Ee,\Ff)$ is a quadratic form on $L^2(K;\nu)$ where the energy $\Ee$ has an explicit expression and the domain $\Ff\subseteq C(K)$. We use the convention to extent $\Ee$ to $L^2(K;\nu)$ as follows. For all $u\in L^2(K;\nu)$, $u$ has at most one continuous version. If $u$ has a continuous version $\tilde{u}$, then we define $\Ee(u,u)$ as the energy of $\tilde{u}$ using its explicit expression which might be $+\infty$, if $u$ has no continuous version, then we define $\Ee(u,u)$ as $+\infty$.

It is obvious that $\Ff_{\beta_1}\supseteq\Ff_{\beta_2}\supseteq\Ff_\loc$ for all $\alpha<\beta_1<\beta_2<\beta^*$. We use Theorem \ref{thm_main} to answer the question about convergence as follows.

\begin{mythm}\label{thm_incre}
For all $u\in L^2(K;\nu)$, we have
$$(5\cdot 2^{-\beta}-1)E_\beta(u,u)\uparrow\Ee_\loc(u,u)$$
as $\beta\uparrow\beta^*=\log5/\log2$.
\end{mythm}

Moreover, we also have a Mosco convergence.

\begin{mythm}\label{thm_conv_main}
For all sequence $\myset{\beta_n}\subseteq(\alpha,\beta^*)$ with $\beta_n\uparrow\beta^*$, we have $(5\cdot2^{-\beta_n}-1)E_{\beta_n}\to\Ee_\loc$ in the sense of Mosco.
\end{mythm}

As a byproduct of Theorem \ref{thm_main}, we obtain the following result about a trace problem. Let us introduce the notion of Besov spaces. Let $(M,d,\mu)$ be a metric measure space and $\alpha,\beta>0$ two parameters. Let
$$B^{2,2}_{\alpha,\beta}(M)=\myset{u\in L^2(M;\mu):\sum_{n=0}^\infty2^{(\alpha+\beta)n}\int\limits_M\int\limits_{d(x,y)<2^{-n}}(u(x)-u(y))^2\mu(\md y)\mu(\md x)<+\infty}.$$
If $\beta>\alpha$, then $B^{2,2}_{\alpha,\beta}(M)$ can be embedded in $C^{\frac{\beta-\alpha}{2}}(M)$. We regard Sierpi\'nski gasket $K$ and unit interval $I$ as metric measure spaces with Euclidean metrics and normalized Hausdorff measures. Let $\alpha_1=\log3/\log2$, $\alpha_2=1$ be the Hausdorff dimensions, $\beta_1^*=\log5/\log2$, $\beta_2^*=2$ the walk dimensions of $K$ and $I$, respectively.

Let us identify $I$ as the segment $[p_0,p_1]\subseteq K$. Choose some $\beta_1\in(\alpha_1,\beta_1^*)$. Any function $u\in B^{2,2}_{\alpha_1,\beta_1}(K)$ is continuous on $K$ and, hence, has the trace $u|_I$ on $I$. The trace problem is the problem of identifying the space of all traces $u|_I$ of all functions $u\in B_{\alpha_1,\beta_1}^{2,2}(K)$. This problem was considered by A. Jonsson using general Besov spaces in $\R^n$, see remarks after \cite[Theorem 3.1]{Jon05}. The following result follows from \cite{Jon05}.

\begin{mythm}\label{thm_trace_main}
Let $\beta_1,\beta_2$ satisfy $\beta_1\in(\alpha_1,\beta_1^*)$ and $\beta_1-\alpha_1=\beta_2-\alpha_2$. Then the trace space of $B^{2,2}_{\alpha_1,\beta_1}(K)$ to $I$ is $B^{2,2}_{\alpha_2,\beta_2}(I)$.
\end{mythm}

We give here a new short proof of Theorem \ref{thm_trace_main} using Theorem \ref{thm_main}.

Finally, we construct explicitly a sequence of non-local Dirichlet forms with jumping kernels equivalent to $|x-y|^{-\alpha-\beta}$ that converges exactly to local Dirichlet form. We need some notions as follows. For all $n\ge1$, $w=w_1\ldots w_n\in W_n$ and $p\in V_w$, we have $p=P_{w_1\ldots w_nw_{n+1}}$ for some $w_{n+1}\in\myset{0,1,2}$. Let $\gamma\ge1$ be an integer, define
$$K^{(i)}_{p,n}=K_{w_1\ldots w_nw_{n+1}\ldots w_{n+1}},i\ge1,$$
with $\gamma ni$ terms of $w_{n+1}$.

\begin{mythm}\label{thm_jumping_kernel}
For all sequence $\myset{\beta_i}\subseteq(\alpha,\beta^*)$ with $\beta_i\uparrow\beta^*$, there exist positive functions $a_i$ bounded from above and below by positive constants given by
$$a_i=\delta_iC_i+(1-\delta_i),$$
where $\myset{\delta_i}\subseteq(0,1)$ is an arbitrary sequence with $\delta_i\uparrow1$ and
$$C_i(x,y)=\sum_{n=1}^{\Phi(i)}2^{-2\alpha n}\sum_{w\in W_n}\sum_{p,q\in V_w}\frac{1}{\nu(K^{(i)}_{p,n})\nu(K^{(i)}_{q,n})}1_{K^{(i)}_{p,n}}(x)1_{K^{(i)}_{q,n}}(y),$$
where $\Phi:\mathbb{N}\to\mathbb{N}$ is increasing and $(5\cdot2^{-\beta_i}-1)\Phi(i)\ge i$ for all $i\ge1$. Then for all $u\in \mathcal{F}_\loc$, we have
$$\lim_{i\to+\infty}(5\cdot 2^{-\beta_i}-1)\iint_{K\times K\backslash\mathrm{diag}}\frac{a_i(x,y)(u(x)-u(y))^2}{|x-y|^{\alpha+\beta_i}}\nu(\md x)\nu(\md y)=\Ee_\loc(u,u).$$
\end{mythm}

\begin{myrmk}
The shape of function $C_i$ reflects the inhomogeneity of fractal structure with respect to Euclidean structure. Of course, subordination technique in \cite{Pie08} ensures the existence of functions $a_i$, but Theorem \ref{thm_jumping_kernel} provides them explicitly.
\end{myrmk}

\section{Proof of Theorem \ref{thm_main}}

First, we give other equivalent semi-norms which are more convenient for later use.
\begin{mylem}\label{lem_equiv1}
For all $u\in L^2(K;\nu)$, we have
$$\int_K\int_K\frac{(u(x)-u(y))^2}{|x-y|^{\alpha+\beta}}\nu(\md x)\nu(\md y)\asymp\sum_{n=0}^\infty2^{(\alpha+\beta)n}\int_K\int_{B(x,2^{-n})}(u(x)-u(y))^2\nu(\md y)\nu(\md x).$$
\end{mylem}
\begin{proof}
On the one hand
$$
\begin{aligned}
&\int_K\int_K\frac{(u(x)-u(y))^2}{|x-y|^{\alpha+\beta}}\nu(\md x)\nu(\md y)\\
&=\int_K\int_{B(x,1)}\frac{(u(x)-u(y))^2}{|x-y|^{\alpha+\beta}}\nu(\md y)\nu(\md x)\\
&=\sum_{n=0}^\infty\int_K\int_{B(x,2^{-n})\backslash{B(x,2^{-(n+1)})}}\frac{(u(x)-u(y))^2}{|x-y|^{\alpha+\beta}}\nu(\md y)\nu(\md x)\\
&\le\sum_{n=0}^\infty2^{(\alpha+\beta)(n+1)}\int_K\int_{B(x,2^{-n})\backslash{B(x,2^{-(n+1)})}}{(u(x)-u(y))^2}\nu(\md y)\nu(\md x)\\
&\le2^{\alpha+\beta}\sum_{n=0}^\infty2^{(\alpha+\beta)n}\int_K\int_{B(x,2^{-n})}{(u(x)-u(y))^2}\nu(\md y)\nu(\md x).
\end{aligned}
$$
On the other hand
$$
\begin{aligned}
&\sum_{n=0}^\infty2^{(\alpha+\beta)n}\int_K\int_{B(x,2^{-n})}{(u(x)-u(y))^2}\nu(\md y)\nu(\md x)\\
&=\sum_{n=0}^\infty\sum_{k=n}^\infty2^{(\alpha+\beta)n}\int_K\int_{B(x,2^{-k})\backslash{B(x,2^{-(k+1)})}}{(u(x)-u(y))^2}\nu(\md y)\nu(\md x)\\
&=\sum_{k=0}^\infty\sum_{n=0}^k 2^{(\alpha+\beta)n}\int_K\int_{B(x,2^{-k})\backslash{B(x,2^{-(k+1)})}}{(u(x)-u(y))^2}\nu(\md y)\nu(\md x)\\
&\le\sum_{k=0}^\infty\frac{2^{(\alpha+\beta)(k+1)}}{2^{\alpha+\beta}-1}\int_K\int_{B(x,2^{-k})\backslash{B(x,2^{-(k+1)})}}{(u(x)-u(y))^2}\nu(\md y)\nu(\md x)\\
&\le\frac{2^{\alpha+\beta}}{2^{\alpha+\beta}-1}\sum_{k=0}^\infty\int_K\int_{B(x,2^{-k})\backslash{B(x,2^{-(k+1)})}}\frac{(u(x)-u(y))^2}{|x-y|^{\alpha+\beta}}\nu(\md y)\nu(\md x)\\
&=\frac{2^{\alpha+\beta}}{2^{\alpha+\beta}-1}\int_K\int_{K}\frac{(u(x)-u(y))^2}{|x-y|^{\alpha+\beta}}\nu(\md x)\nu(\md y).
\end{aligned}
$$
\end{proof}

Moreover, we have

\begin{mycor}\label{cor_arbi}
Fix arbitrary integer $N\ge0$ and real number $c>0$. For all $u\in L^2(K;\nu)$, we have
$$\int_K\int_K\frac{(u(x)-u(y))^2}{|x-y|^{\alpha+\beta}}\nu(\md x)\nu(\md y)\asymp\sum_{n=N}^\infty2^{(\alpha+\beta)n}\int_K\int_{B(x,c2^{-n})}(u(x)-u(y))^2\nu(\md y)\nu(\md x).$$
\end{mycor}

\begin{proof}
We only need to show that for all $n\ge1$, there exists some positive constant $C=C(n)$ such that
$$\int_K\int_K(u(x)-u(y))^2\nu(\md x)\nu(\md y)\le C\int_K\int_{B(x,{2^{-n}})}(u(x)-u(y))^2\nu(\md y)\nu(\md x).$$
Indeed, since SG satisfies the chain condition, see \cite[Definition 3.4]{GHL03}, that is, there exists a positive constant $C_1$ such that for all $x,y\in K$, for all integer $N\ge1$ there exist $z_0,\ldots,z_N\in K$ with $z_0=x,z_N=y$ and
$$|z_i-z_{i+1}|\le C_1\frac{|x-y|}{N}\text{ for all }i=0,\ldots,N-1.$$
Take integer $N\ge2^{n+2}C_1+1$. Fix $x,y\in K$, there exist $z_0,\ldots,z_N$ with $z_0=x,z_N=y$ and
$$|z_i-z_{i+1}|\le C_1\frac{|x-y|}{N}\le 2^{-(n+2)}\text{ for all }i=0,\ldots,N-1.$$
For all $i=0,\ldots,N-1$, for all $x_i\in B(z_i,2^{-(n+2)})$, $x_{i+1}\in B(z_{i+1},2^{-(n+2)})$, we have
$$|x_i-x_{i+1}|\le|x_i-z_i|+|z_i-z_{i+1}|+|z_{i+1}-x_{i+1}|\le{3}\cdot{2^{-(n+2)}}<2^{-n}.$$
Fix $x_0=z_0=x$, $x_N=z_N=y$, note that
$$(u(x)-u(y))^2=(u(x_0)-u(x_N))^2\le N\sum_{i=0}^{N-1}(u(x_i)-u(x_{i+1}))^2.$$
Integrating with respect to $x_1\in B(z_1,2^{-(n+2)}),\ldots,x_{N-1}\in B(z_{N-1},2^{-(n+2)})$ and dividing by $\nu(B(z_1,2^{-(n+2)})),\ldots,\nu(B(z_{N-1},2^{-(n+2)}))$, we have
$$
\begin{aligned}
(u(x)-u(y))^2&\le N\left(\frac{1}{\nu(B(z_1,2^{-(n+2)}))}\int_{B(z_1,2^{-(n+2)})}(u(x_0)-u(x_1))^2\nu(\md x_1)\right.\\
&+\frac{1}{\nu(B(z_{N-1},2^{-(n+2)}))}\int_{B(z_{N-1},2^{-(n+2)})}(u(x_{N-1})-u(x_N))^2\nu(\md x_{N-1})\\
&+\sum_{i=1}^{N-2}\frac{1}{\nu(B(z_i,2^{-(n+2)}))\nu(B(z_{i+1},2^{-(n+2)}))}\\
&\left.\int_{B(z_i,2^{-(n+2)})}\int_{B(z_{i+1},2^{-(n+2)})}(u(x_i)-u(x_{i+1}))^2\nu(\md x_i)\nu(\md x_{i+1})\right).
\end{aligned}
$$
Noting that $\nu(B(z_i,2^{-(n+2)}))\asymp 2^{-\alpha n}$ for all $i=1,\ldots,N-1$, we have
$$
\begin{aligned}
(u(x)-u(y))^2&\le C_2\left(\int_{B(x,2^{-n})}(u(x)-u(x_1))^2\nu(\md x_1)\right.\\
&+\int_{B(y,2^{-n})}(u(y)-u(x_{N-1}))^2\nu(\md x_{N-1})\\
&\left.+\int_K\int_{B(x,2^{-n})}(u(x)-u(y))^2\nu(\md y)\nu(\md x)\right),
\end{aligned}
$$
where $C_2=C_2(n)$ is some positive constant. Since $\nu(K)=1$, integrating with respect to $x,y\in K$, we have
$$\int_K\int_K(u(x)-u(y))^2\nu(\md x)\nu(\md y)\le 4C_2\int_K\int_{B(x,{2^{-n}})}(u(x)-u(y))^2\nu(\md y)\nu(\md x).$$
Letting $C=4C_2$, then we have desired result.
\end{proof}

The following result states that a Besov space can be embedded in some H\"older space.
\begin{mylem}\label{lem_holder}(\cite[Theorem 4.11 (iii)]{GHL03})
Let $u\in C(K)$ and
$$E(u):=\int_K\int_K\frac{(u(x)-u(y))^2}{|x-y|^{\alpha+\beta}}\nu(\md x)\nu(\md y),$$
then
$$|u(x)-u(y)|^2\le cE(u)|x-y|^{\beta-\alpha}\text{ for all }x,y\in K,$$
where $c$ is some positive constant.
\end{mylem}

Note that the proof of above lemma does not rely on heat kernel.

We divide Theorem \ref{thm_main} into the following Theorem \ref{thm_equiv2_1} and Theorem \ref{thm_equiv2_2}. The idea of the proofs of these theorems comes from \cite{Jon96} where the case of local Dirichlet form was considered.

\begin{mythm}\label{thm_equiv2_1}
For all $u\in C(K)$, we have
$$\sum_{n=1}^\infty2^{(\beta-\alpha)n}\sum_{w\in W_n}\sum_{p,q\in V_w}(u(p)-u(q))^2\lesssim\int_K\int_K\frac{(u(x)-u(y))^2}{|x-y|^{\alpha+\beta}}\nu(\md x)\nu(\md y).$$
\end{mythm}

\begin{proof}
First fix $n\ge1$ and $w=w_1\ldots w_n\in W_n$, consider $\sum_{p,q\in V_w}(u(p)-u(q))^2$. For all $x\in K_w$, we have
$$(u(p)-u(q))^2\le 2(u(p)-u(x))^2+2(u(x)-u(q))^2.$$
Integrating with respect to $x\in K_w$ and dividing by $\nu(K_w)$, we have
$$(u(p)-u(q))^2\le\frac{2}{\nu(K_w)}\int_{K_w}(u(p)-u(x))^2\nu(\md x)+\frac{2}{\nu(K_w)}\int_{K_w}(u(q)-u(x))^2\nu(\md x),$$
hence
$$
\begin{aligned}
&\sum_{p,q\in V_w}(u(p)-u(q))^2\\
&\le\sum_{p,q\in V_w,p\ne q}\left[\frac{2}{\nu(K_w)}\int_{K_w}(u(p)-u(x))^2\nu(\md x)+\frac{2}{\nu(K_w)}\int_{K_w}(u(q)-u(x))^2\nu(\md x)\right]\\
&\le2\cdot2\cdot2\sum_{p\in V_w}\frac{1}{\nu(K_w)}\int_{K_w}(u(p)-u(x))^2\nu(\md x).
\end{aligned}
$$
Consider $(u(p)-u(x))^2$, $p\in V_w$, $x\in K_w$. There exists $w_{n+1}\in\myset{0,1,2}$ such that $p=f_{w_1}\circ\ldots\circ f_{w_n}(p_{w_{n+1}})$. Let $k,l\ge1$ be integers to be determined later, let
$$w^{(i)}=w_1\ldots w_nw_{n+1}\ldots w_{n+1}$$
with $ki$ terms of $w_{n+1}$, $i=0,\ldots,l$. For all $x^{(i)}\in K_{w^{(i)}}$, $i=0,\ldots,l$, we have
$$
\begin{aligned}
(u(p)-u(x^{(0)}))^2&\le2(u(p)-u(x^{(l)}))^2+2(u(x^{(0)})-u(x^{(l)}))^2\\
&\le2(u(p)-u(x^{(l)}))^2+2\left[2(u(x^{(0)})-u(x^{(1)}))^2+2(u(x^{(1)})-u(x^{(l)}))^2\right]\\
&=2(u(p)-u(x^{(l)}))^2+2^2(u(x^{(0)})-u(x^{(1)}))^2+2^2(u(x^{(1)})-u(x^{(l)}))^2\\
&\le\ldots\le2(u(p)-u(x^{(l)}))^2+2^2\sum_{i=0}^{l-1}2^i(u(x^{(i)})-u(x^{(i+1)}))^2.
\end{aligned}
$$
Integrating with respect to $x^{(0)}\in K_{w^{(0)}}$, \ldots, $x^{(l)}\in K_{w^{(l)}}$ and dividing by $\nu(K_{w^{(0)}})$, \ldots, $\nu(K_{w^{(l)}})$, we have
$$
\begin{aligned}
&\frac{1}{\nu(K_{w^{(0)}})}\int_{K_{w^{(0)}}}(u(p)-u(x^{(0)}))^2\nu(\md x^{(0)})\\
&\le\frac{2}{\nu(K_{w^{(l)}})}\int_{K_{w^{(l)}}}(u(p)-u(x^{(l)}))^2\nu(\md x^{(l)})\\
&+2^2\sum_{i=0}^{l-1}\frac{2^i}{\nu(K_{w^{(i)}})\nu(K_{w^{(i+1)}})}\int_{K_{w^{(i)}}}\int_{K_{w^{(i+1)}}}(u(x^{(i)})-u(x^{(i+1)}))^2\nu(\md x^{(i)})\nu(\md x^{(i+1)}).
\end{aligned}
$$
Now let us use $\nu(K_{w^{(i)}})=(1/3)^{n+ki}=2^{-\alpha(n+ki)}$. For the first term, by Lemma \ref{lem_holder}, we have
$$
\begin{aligned}
\frac{1}{\nu(K_{w^{(l)}})}\int_{K_{w^{(l)}}}(u(p)-u(x^{(l)}))^2\nu(\md x^{(l)})&\le \frac{cE(u)}{\nu(K_{w^{(l)}})}\int_{K_{w^{(l)}}}|p-x^{(l)}|^{\beta-\alpha}\nu(\md x^{(l)})\\
&\le cE(u){2}^{-(\beta-\alpha)(n+kl)}.
\end{aligned}
$$
For the second term, for all $x^{(i)}\in K_{w^{(i)}},x^{(i+1)}\in K_{w^{(i+1)}}$, we have $|x^{(i)}-x^{(i+1)}|\le2^{-(n+ki)}$, hence
$$
\begin{aligned}
&\sum_{i=0}^{l-1}\frac{2^i}{\nu(K_{w^{(i)}})\nu(K_{w^{(i+1)}})}\int_{K_{w^{(i)}}}\int_{K_{w^{(i+1)}}}(u(x^{(i)})-u(x^{(i+1)}))^2\nu(\md x^{(i)})\nu(\md x^{(i+1)})\\
&\le\sum_{i=0}^{l-1}{2^{i+\alpha(n+ki+n+ki+k)}}\int_{K_{w^{(i)}}}\int_{|x^{(i+1)}-x^{(i)}|\le2^{-n-ki}}(u(x^{(i)})-u(x^{(i+1)}))^2\nu(\md x^{(i)})\nu(\md x^{(i+1)})\\
&=\sum_{i=0}^{l-1}{2^{i+\alpha k+2\alpha(n+ki)}}\int_{K_{w^{(i)}}}\int_{|x^{(i+1)}-x^{(i)}|\le2^{-(n+ki)}}(u(x^{(i)})-u(x^{(i+1)}))^2\nu(\md x^{(i)})\nu(\md x^{(i+1)}),
\end{aligned}
$$
and
$$
\begin{aligned}
&\frac{1}{\nu(K_w)}\int_{K_w}(u(p)-u(x))^2\nu(\md x)=\frac{1}{\nu(K_{w^{(0)}})}\int_{K_{w^{(0)}}}(u(p)-u(x^{(0)}))^2\nu(\md x^{(0)})\\
&\le 2cE(u)2^{-(\beta-\alpha)(n+kl)}\\
&+4\sum_{i=0}^{l-1}{2^{i+\alpha k+2\alpha(n+ki)}}\int_{K_{w^{(i)}}}\int_{|x^{(i+1)}-x^{(i)}|\le2^{-n-ki}}(u(x^{(i)})-u(x^{(i+1)}))^2\nu(\md x^{(i)})\nu(\md x^{(i+1)}).
\end{aligned}
$$
Hence
$$
\begin{aligned}
&\sum_{w\in W_n}\sum_{p,q\in V_w}(u(p)-u(q))^2\\
&\le8\sum_{w\in W_n}\sum_{p\in V_w}\frac{1}{\nu(K_w)}\int_{K_w}(u(p)-u(x))^2\nu(\md x)\\
&\le8\sum_{w\in W_n}\sum_{p\in V_w}\left(2cE(u)2^{-(\beta-\alpha)(n+kl)}\right.\\
&\left.+4\sum_{i=0}^{l-1}{2^{i+\alpha k+2\alpha(n+ki)}}\int_{K_{w^{(i)}}}\int_{|x^{(i+1)}-x^{(i)}|\le2^{-n-ki}}(u(x^{(i)})-u(x^{(i+1)}))^2\nu(\md x^{(i)})\nu(\md x^{(i+1)})\right).
\end{aligned}
$$
For the first term, we have
$$\sum_{w\in W_n}\sum_{p\in V_w}2^{-(\beta-\alpha)(n+kl)}=3\cdot3^n\cdot2^{-(\beta-\alpha)(n+kl)}=3\cdot2^{\alpha n-(\beta-\alpha)(n+kl)}.$$
For the second term, fix $i=0,\ldots,l-1$, different $p\in V_w$, $w\in W_n$ correspond to different $K_{w^{(i)}}$, hence
$$
\begin{aligned}
&\sum_{i=0}^{l-1}\sum_{w\in W_n}\sum_{p\in V_w}2^{i+\alpha k+2\alpha(n+ki)}\\
&\cdot\int_{K_{w^{(i)}}}\int_{|x^{(i+1)}-x^{(i)}|\le2^{-n-ki}}(u(x^{(i)})-u(x^{(i+1)}))^2\nu(\md x^{(i)})\nu(\md x^{(i+1)})\\
&\le\sum_{i=0}^{l-1}2^{i+\alpha k+2\alpha(n+ki)}\int_{K}\int_{|x^{(i+1)}-x^{(i)}|\le2^{-(n+ki)}}(u(x^{(i)})-u(x^{(i+1)}))^2\nu(\md x^{(i)})\nu(\md x^{(i+1)})\\
&=2^{\alpha k}\sum_{i=0}^{l-1}2^{i-(\beta-\alpha)ki-(\beta-\alpha)n}\\
&\cdot\left(2^{(\alpha+\beta)(n+ki)}\int_{K}\int_{|x^{(i+1)}-x^{(i)}|\le2^{-(n+ki)}}(u(x^{(i)})-u(x^{(i+1)}))^2\nu(\md x^{(i)})\nu(\md x^{(i+1)})\right).
\end{aligned}
$$
For simplicity, denote
$$E_{n}(u)=2^{(\alpha+\beta)n}\int_{K}\int_{|x-y|\le2^{-n}}(u(x)-u(y))^2\nu(\md x)\nu(\md y).$$
We have
$$
\begin{aligned}
&\sum_{w\in W_n}\sum_{p,q\in V_w}(u(p)-u(q))^2\\
&\le48cE(u)\cdot2^{\alpha n-(\beta-\alpha)(n+kl)}+32\cdot2^{\alpha k}\sum_{i=0}^{l-1}2^{i-(\beta-\alpha)ki-(\beta-\alpha)n}E_{n+ki}(u).
\end{aligned}
$$
Hence
$$
\begin{aligned}
&\sum_{n=1}^\infty2^{(\beta-\alpha)n}\sum_{w\in W_n}\sum_{p,q\in V_w}(u(p)-u(q))^2\\
&\le48cE(u)\sum_{n=1}^\infty2^{\beta n-(\beta-\alpha)(n+kl)}+32\cdot2^{\alpha k}\sum_{n=1}^\infty\sum_{i=0}^{l-1}2^{i-(\beta-\alpha)ki}E_{n+ki}(u).
\end{aligned}
$$
Take $l=n$, then
$$
\begin{aligned}
&\sum_{n=1}^\infty2^{(\beta-\alpha)n}\sum_{w\in W_n}\sum_{p,q\in V_w}(u(p)-u(q))^2\\
&\le48cE(u)\sum_{n=1}^\infty2^{[\beta-(\beta-\alpha)(k+1)]n}+32\cdot2^{\alpha k}\sum_{n=1}^\infty\sum_{i=0}^{n-1}2^{i-(\beta-\alpha)ki}E_{n+ki}(u)\\
&=48cE(u)\sum_{n=1}^\infty2^{[\beta-(\beta-\alpha)(k+1)]n}+32\cdot2^{\alpha k}\sum_{i=0}^\infty2^{i-(\beta-\alpha)ki}\sum_{n=i+1}^\infty E_{n+ki}(u)\\
&\le48cE(u)\sum_{n=1}^\infty2^{[\beta-(\beta-\alpha)(k+1)]n}+32\cdot2^{\alpha k}C_1E(u)\sum_{i=0}^\infty 2^{[1-(\beta-\alpha)k]i},
\end{aligned}
$$
where $C_1$ is some positive constant from Lemma \ref{lem_equiv1}. Take $k\ge1$ such that $\beta-(\beta-\alpha)(k+1)<0$ and $1-(\beta-\alpha)k<0$, then above two series converge, hence
$$\sum_{n=1}^\infty2^{(\beta-\alpha)n}\sum_{w\in W_n}\sum_{p,q\in V_w}(u(p)-u(q))^2\lesssim\int_K\int_K\frac{(u(x)-u(y))^2}{|x-y|^{\alpha+\beta}}\nu(\md x)\nu(\md y).$$
\end{proof}

\begin{mythm}\label{thm_equiv2_2}
For all $u\in C(K)$, we have
\begin{equation}\label{eqn_equiv2_1}
\int_K\int_K\frac{(u(x)-u(y))^2}{|x-y|^{\alpha+\beta}}\nu(\md x)\nu(\md y)\lesssim\sum_{n=1}^\infty2^{(\beta-\alpha)n}\sum_{w\in W_n}\sum_{p,q\in V_w}(u(p)-u(q))^2,
\end{equation}
or equivalently
\begin{equation}\label{eqn_equiv2_2}
\sum_{n=1}^\infty2^{(\alpha+\beta)n}\int\limits_K\int\limits_{B(x,2^{-n-1})}(u(x)-u(y))^2\nu(\md y)\nu(\md x)\lesssim\sum_{n=1}^\infty2^{(\beta-\alpha)n}\sum_{w\in W_n}\sum_{p,q\in V_w}(u(p)-u(q))^2.
\end{equation}
\end{mythm}

\begin{proof}
Note $V_n=\cup_{w\in W_n}V_w$, it is obvious that its cardinal $\#V_n\asymp3^n=2^{\alpha n}$. Let $\nu_n$ be the measure on $V_n$ which assigns $1/\#V_n$ on each point of $V_n$, then $\nu_n$ converges weakly to $\nu$.

First, fix $n\ge1$ and $m\ge n$, we estimate
$$2^{(\alpha+\beta)n}\int_K\int_{B(x,2^{-n-1})}(u(x)-u(y))^2\nu_m(\md y)\nu_m(\md x).$$
Note that
$$
\begin{aligned}
&\int_K\int_{B(x,2^{-n-1})}(u(x)-u(y))^2\nu_m(\md y)\nu_m(\md x)\\
&=\sum_{w\in W_n}\int_{K_w}\int_{B(x,2^{-n-1})}(u(x)-u(y))^2\nu_m(\md y)\nu_m(\md x).
\end{aligned}
$$
Fix $w\in W_n$, there exist at most four $\tilde{w}\in W_n$ such that $K_{\tilde{w}}\cap K_w\ne\emptyset$, let
$$K_w^*=\cup_{\tilde{w}\in W_n,K_{\tilde{w}}\cap K_w\ne\emptyset}K_{\tilde{w}}.$$
For all $x\in K_w$, $y\in B(x,2^{-n-1})$, we have $y\in K_w^*$, hence
$$\int_{K_w}\int_{B(x,2^{-n-1})}(u(x)-u(y))^2\nu_m(\md y)\nu_m(\md x)\le\int_{K_w}\int_{K_w^*}(u(x)-u(y))^2\nu_m(\md y)\nu_m(\md x).$$
For all $x\in K_w$, $y\in K_w^*$, there exists $\tilde{w}\in W_n$ such that $y\in K_{\tilde{w}}$ and $K_{\tilde{w}}\cap K_w\ne\emptyset$. Take $z\in V_w\cap V_{\tilde{w}}$, then
$$(u(x)-u(y))^2\le2(u(x)-u(z))^2+2(u(z)-u(y))^2,$$
and
$$
\begin{aligned}
&\int_{K_w}\int_{K_w^*}(u(x)-u(y))^2\nu_m(\md y)\nu_m(\md x)\\
&\le\sum_{\tilde{w}\in W_n,K_{\tilde{w}}\cap K_w\ne\emptyset}\int_{K_w}\int_{K_{\tilde{w}}}(u(x)-u(y))^2\nu_m(\md y)\nu_m(\md x)\\
&\le\sum_{\tilde{w}\in W_n,K_{\tilde{w}}\cap K_w\ne\emptyset}2\int_{K_w}\int_{K_{\tilde{w}}}\left((u(x)-u(z))^2+(u(z)-u(y))^2\right)\nu_m(\md y)\nu_m(\md x).
\end{aligned}
$$
Hence
\begin{equation}\label{eqn_tmp1}
\begin{aligned}
&\sum_{w\in W_n}\int_{K_w}\int_{K_w^*}(u(x)-u(y))^2\nu_m(\md y)\nu_m(\md x)\\
&\le2\cdot2\cdot4\cdot2\sum_{w\in W_n}\sum_{z\in V_w}\int_{K_w}(u(x)-u(z))^2\nu_m(\md x)\left(\int_{K_w}\nu_m(\md y)\right)\\
&=32\sum_{w\in W_n}\sum_{z\in V_w}\int_{K_w}(u(x)-u(z))^2\nu_m(\md x)\frac{\#(V_m\cap K_w)}{\#V_m}\\
&=32\sum_{w\in W_n}\sum_{z\in V_w}\sum_{x\in V_m\cap K_w}(u(x)-u(z))^2\frac{1}{\#V_m}\frac{\#(V_m\cap K_w)}{\#V_m}\\
&=32\frac{\#V_{m-n}}{(\#V_m)^2}\sum_{w\in W_n}\sum_{z\in V_w}\sum_{x\in V_m\cap K_w}(u(x)-u(z))^2.
\end{aligned}
\end{equation}
Let us estimate $(u(x)-u(z))^2$ for $z\in V_w$, $x\in V_m\cap K_w$, $w\in W_n$. We construct a finite sequence $p_n,\ldots,p_{m+1}$ as follows. If $w=w_1\ldots w_n\in W_n$, then
$$
\begin{aligned}
z&=P_{w_1\ldots w_nw_{n+1}},\\
x&=P_{w_1\ldots w_n\tilde{w}_{n+1}\ldots\tilde{w}_m\tilde{w}_{m+1}}.
\end{aligned}
$$
Let
$$
\begin{aligned}
p_n&=P_{w_1\ldots w_nw_{n+1}}=z,\\
p_{n+1}&=P_{w_1\ldots w_n\tilde{w}_{n+1}},\\
p_{n+2}&=P_{w_1\ldots w_n\tilde{w}_{n+1}\tilde{w}_{n+2}},\\
&\ldots\\
p_{m+1}&=P_{w_1\ldots w_n\tilde{w}_{n+1}\ldots\tilde{w}_{m}\tilde{w}_{m+1}}=x,
\end{aligned}
$$
then $|p_i-p_{i+1}|=0$ or $2^{-i}$, $i=n,\ldots,m$ and
$$
\begin{aligned}
&(u(x)-u(z))^2\\
&=(u(p_n)-u(p_{m+1}))^2\le2(u(p_n)-u(p_{n+1}))^2+2(u(p_{n+1})-u(p_{m+1}))^2\\
&\le2(u(p_n)-u(p_{n+1}))^2+2\left[2(u(p_{n+1})-u(p_{n+2}))^2+2(u(p_{n+2})-u(p_{m+1}))^2\right]\\
&=2(u(p_n)-u(p_{n+1}))^2+2^2(u(p_{n+1})-u(p_{n+2}))^2+2^2(u(p_{n+2})-u(p_{m+1}))^2\\
&\le\ldots\le\sum_{i=n}^{m}2^{i-n+1}(u(p_i)-u(p_{i+1}))^2.
\end{aligned}
$$
Let us sum up the resulting inequality for all $z\in V_w$, $x\in V_m\cap K_w$, $w\in W_n$. For all $i=n,\ldots,m$, $p,q\in V_i\cap K_w$ with $|p-q|=2^{-i}$, the term $(u(p)-u(q))^2$ occurs in the sum with times of the order $3^{m-i}$, hence
$$\sum_{w\in W_n}\sum_{z\in V_w}\sum_{x\in V_m\cap K_w}(u(x)-u(z))^2\le c\sum_{i=n}^m\sum_{w\in W_i}\sum_{p,q\in V_w}(u(p)-u(q))^2\cdot3^{m-i}\cdot2^{i-n}.$$
It follows from Equation (\ref{eqn_tmp1}) that
$$
\begin{aligned}
&\sum_{w\in W_n}\int_{K_w}\int_{K_{{w}}^*}(u(x)-u(y))^2\nu_m(\md y)\nu_m(\md x)\\
&\le c\frac{3^{m-n}}{3^{2m}}\sum_{i=n}^m\sum_{w\in W_i}\sum_{p,q\in V_w}(u(p)-u(q))^2\cdot3^{m-i}\cdot2^{i-n}\\
&=c\sum_{i=n}^m\sum_{w\in W_i}\sum_{p,q\in V_w}(u(p)-u(q))^2\cdot 3^{-n-i}\cdot2^{i-n}.
\end{aligned}
$$
Letting $m\to+\infty$, we obtain
$$\int_K\int_{B(x,2^{-n-1})}(u(x)-u(y))^2\nu(\md y)\nu(\md x)\le c\sum_{i=n}^\infty\sum_{w\in W_i}\sum_{p,q\in V_w}(u(p)-u(q))^2\cdot 3^{-n-i}\cdot2^{i-n},$$
and
$$
\begin{aligned}
&2^{(\alpha+\beta)n}\int_K\int_{B(x,2^{-n-1})}(u(x)-u(y))^2\nu(\md y)\nu(\md x)\\
&\le c\sum_{i=n}^\infty\sum_{w\in W_i}\sum_{p,q\in V_w}(u(p)-u(q))^2\cdot 2^{-(\alpha-1)i}\cdot2^{(\beta-1)n},
\end{aligned}
$$
and hence
$$
\begin{aligned}
&\sum_{n=1}^\infty2^{(\alpha+\beta)n}\int_K\int_{B(x,2^{-n-1})}(u(x)-u(y))^2\nu(\md y)\nu(\md x)\\
&\le c\sum_{n=1}^\infty\sum_{i=n}^\infty\sum_{w\in W_i}\sum_{p,q\in V_w}(u(p)-u(q))^2\cdot 2^{-(\alpha-1)i}\cdot2^{(\beta-1)n}\\
&=c\sum_{i=1}^\infty\sum_{n=1}^i\sum_{w\in W_i}\sum_{p,q\in V_w}(u(p)-u(q))^2\cdot 2^{-(\alpha-1)i}\cdot2^{(\beta-1)n}\\
&\le\frac{2^{\beta-1}c}{2^{\beta-1}-1}\sum_{i=1}^\infty2^{(\beta-\alpha)i}\sum_{w\in W_i}\sum_{p,q\in V_w}(u(p)-u(q))^2,
\end{aligned}
$$
which proves Equation (\ref{eqn_equiv2_2}). Applying Corollary \ref{cor_arbi}, we obtain Equation (\ref{eqn_equiv2_1}).
\end{proof}

\section{Proof of Theorem \ref{thm_incre} and \ref{thm_conv_main}}

For simplicity, let $\lambda=2^{-\beta}$ or $\beta=-\log\lambda/\log2$, where $\beta\in(\alpha,\beta^*)$ or $\lambda\in(1/5,1/3)$, write
$$
\mathscr{E}_\lambda(u,u)=(5\cdot 2^{-\beta}-1)E_\beta(u,u)=(5\lambda-1)\sum_{n=1}^\infty\frac{1}{(5\lambda)^n}\left[\left(\frac{5}{3}\right)^n\sum_{w\in W_n}\sum_{p,q\in V_w}(u(p)-u(q))^2\right].
$$
First, we prove Theorem \ref{thm_incre}.
\begin{proof}[Proof of Theorem \ref{thm_incre}]
If $u$ has no continuous version, then this result is obvious. Hence, we may assume that $u$ is continuous. Let
$$a_n=a_n(u)=\left(\frac{5}{3}\right)^n\sum_{w\in W_n}\sum_{p,q\in V_w}(u(p)-u(q))^2,$$
then
$$a_n\uparrow a_\infty=a_\infty(u)=\lim_{n\to+\infty}\left(\frac{5}{3}\right)^n\sum_{w\in W_n}\sum_{p,q\in V_w}(u(p)-u(q))^2.$$
First, we show that $\mathscr{E}_\lambda(u,u)\to\Ee_\loc(u,u)$ as $\lambda\downarrow1/5$, that is,
$$\lim_{\lambda\downarrow1/5}(5\lambda-1)\sum_{n=1}^\infty\frac{1}{(5\lambda)^n}a_n=a_\infty.$$
Note that
$$(5\lambda-1)\sum_{n=1}^\infty\frac{1}{(5\lambda)^n}a_n\le(5\lambda-1)\sum_{n=1}^\infty\frac{1}{(5\lambda)^n}a_\infty=a_\infty,$$
we have
$$\varlimsup_{\lambda\downarrow1/5}(5\lambda-1)\sum_{n=1}^\infty\frac{1}{(5\lambda)^n}a_n\le a_\infty.$$
On the other hand, for all $A<a_\infty$, there exists $N\ge1$ such that for all $n>N$, we have $a_n>A$, hence
$$
\begin{aligned}
(5\lambda-1)\sum_{n=1}^\infty\frac{1}{(5\lambda)^n}a_n&\ge(5\lambda-1)\sum_{n=N+1}^\infty\frac{1}{(5\lambda)^n}A=(5\lambda-1)\frac{\frac{1}{(5\lambda)^{N+1}}}{1-\frac{1}{5\lambda}}A\\
&=(5\lambda-1)\frac{\frac{1}{(5\lambda)^{N}}}{5\lambda-1}A=\frac{1}{(5\lambda)^N}A\to A,
\end{aligned}
$$
as $\lambda\downarrow1/5$, hence
$$\varliminf_{\lambda\downarrow1/5}(5\lambda-1)\sum_{n=1}^\infty\frac{1}{(5\lambda)^n}a_n\ge A,$$
for all $A<a_\infty$, hence
$$\varliminf_{\lambda\downarrow1/5}(5\lambda-1)\sum_{n=1}^\infty\frac{1}{(5\lambda)^n}a_n\ge a_\infty.$$
We have
$$\lim_{\lambda\downarrow1/5}(5\lambda-1)\sum_{n=1}^\infty\frac{1}{(5\lambda)^n}a_n=a_\infty.$$
If $\mathscr{E}_\lambda(u,u)<+\infty$, then we have
$$\sum_{n=1}^\infty\frac{1}{(5\lambda)^n}a_n<+\infty.$$
Hence
$$
\begin{aligned}
(5\lambda-1)\sum_{n=1}^\infty\frac{1}{(5\lambda)^n}a_n&=5\lambda\sum_{n=1}^\infty\frac{1}{(5\lambda)^n}a_n-\sum_{n=1}^\infty\frac{1}{(5\lambda)^n}a_n=\sum_{n=1}^\infty\frac{1}{(5\lambda)^{n-1}}a_n-\sum_{n=1}^\infty\frac{1}{(5\lambda)^n}a_n\\
&=\sum_{n=0}^\infty\frac{1}{(5\lambda)^n}a_{n+1}-\sum_{n=1}^\infty\frac{1}{(5\lambda)^n}a_n=a_1+\sum_{n=1}^\infty\frac{1}{(5\lambda)^n}(a_{n+1}-a_n).
\end{aligned}
$$
Assume that $1/3>\lambda_1>\lambda_2>1/5$ and observe the following
\begin{itemize}
\item If $\mathscr{E}_{\lambda_2}(u,u)=+\infty$, then it is obvious that $\mathscr{E}_{\lambda_2}(u,u)\ge\mathscr{E}_{\lambda_1}(u,u)$.
\item If $\mathscr{E}_{\lambda_2}(u,u)<+\infty$, then we have $\mathscr{E}_{\lambda_1}(u,u)<+\infty$, hence
$$\mathscr{E}_{\lambda_1}(u,u)=a_1+\sum_{n=1}^\infty\frac{1}{(5\lambda_1)^n}(a_{n+1}-a_n)\le a_1+\sum_{n=1}^\infty\frac{1}{(5\lambda_2)^n}(a_{n+1}-a_n)=\mathscr{E}_{\lambda_2}(u,u).$$
\end{itemize}
Therefore, $\mathscr{E}_{\lambda_2}(u,u)\ge\mathscr{E}_{\lambda_1}(u,u)$ and $\mathscr{E}_\lambda(u,u)\uparrow\Ee_\loc(u,u)$ as $\lambda\downarrow1/5$.
\end{proof}

In what follows, $K$ is a locally compact separable metric space and $\nu$ is a Radon measure on $K$ with full support. If $(\Ee,\Ff)$ is a closed form on $L^2(K;\nu)$, we extend $\Ee$ to be $+\infty$ outside $\Ff$, hence the information of $\Ff$ is encoded in $\Ee$.

\begin{mydef}\label{def_Mosco}
Let $\Ee^n$, $\Ee$ be closed forms on $L^2(K;\nu)$. We say that $\Ee^n$ converges to $\Ee$ in the sense of Mosco if the following conditions are satisfied.
\begin{enumerate}[(1)]
\item\label{def_Mosco_1} For all $\myset{u_n}\subseteq L^2(K;\nu)$ that converges weakly to $u\in L^2(K;\nu)$, we have
$$\varliminf_{n\to+\infty}\Ee^n(u_n,u_n)\ge\Ee(u,u).$$
\item\label{def_Mosco_2} For all $u\in L^2(K;\nu)$, there exists a sequence $\myset{u_n}\subseteq L^2(K;\nu)$ converging strongly to $u$ in $L^2(K;\nu)$ such that
$$\varlimsup_{n\to+\infty}\Ee^n(u_n,u_n)\le\Ee(u,u).$$
\end{enumerate}
\end{mydef}

Let $\myset{P_t:t>0}$, $\myset{P^n_t:t>0}$ be the semigroups and $\myset{G_\alpha:\alpha>0}$, $\myset{G^n_\alpha:\alpha>0}$ the resolvents corresponding to $\Ee$, $\Ee^n$. We have the following equivalence.

\begin{myprop}(\cite[Theorem 2.4.1, Corollary 2.6.1]{Mos94})\label{prop_Mosco}
The followings are equivalent.
\begin{enumerate}[(1)]
\item $\Ee^n$ converges to $\Ee$ in the sense of Mosco.
\item $P^n_tu\to P_tu$ in $L^2(K;\nu)$ for all $t>0$, $u\in L^2(K;\nu)$.
\item $G^n_\alpha u\to G_\alpha u$ in $L^2(K;\nu)$ for all $\alpha>0$, $u\in L^2(K;\nu)$.
\end{enumerate}
\end{myprop}

We have following corollary.

\begin{mycor}\label{cor_Mosco}
Let $(\Ee,\Ff)$ be a closed form on $L^2(K;\nu)$, then for all $\myset{u_n}\subseteq L^2(K;\nu)$ that converges weakly to $u\in L^2(K;\nu)$, we have
\begin{equation}\label{eqn_liminf}
\Ee(u,u)\le\varliminf_{n\to+\infty}\Ee(u_n,u_n).
\end{equation}
\end{mycor}

\begin{proof}
Let $\Ee^n=\Ee$ for all $n\ge1$, then by Proposition \ref{prop_Mosco}, $\Ee^n$ is trivially convergent to $\Ee$ in the sense of Mosco. By definition, Equation (\ref{eqn_liminf}) is obvious.
\end{proof}
Note that it will be tedious to prove Corollary \ref{cor_Mosco} directly.

In what follows, $K$ is SG in $\R^2$ and $\nu$ is the normalized Hausdorff measure on $K$.
\begin{proof}[Proof of Theorem \ref{thm_conv_main}]
First, we check condition (\ref{def_Mosco_2}). For all $u\in L^2(K;\nu)$, let $u_n=u$ for all $n\ge1$, then $u_n$ is trivially convergent to $u$ in $L^2(K;\nu)$ and by Theorem \ref{thm_incre}, we have
$$\Ee_\loc(u,u)=\lim_{n\to+\infty}\mathscr{E}_{\lambda_n}(u,u)=\lim_{n\to+\infty}\mathscr{E}_{\lambda_n}(u_n,u_n).$$
Then, we check condition (\ref{def_Mosco_1}). For all $\myset{u_n}\subseteq L^2(K;\nu)$ that converges weakly to $u\in L^2(K;\nu)$. For all $m\ge1$, by Corollary \ref{cor_Mosco}, we have
$$\mathscr{E}_{\lambda_m}(u,u)\le\varliminf_{n\to+\infty}\mathscr{E}_{\lambda_m}(u_n,u_n),$$
by Theorem \ref{thm_incre}, for all $n\ge m$, we have $\mathscr{E}_{\lambda_m}(u_n,u_n)\le\mathscr{E}_{\lambda_n}(u_n,u_n)$, hence
$$\mathscr{E}_{\lambda_m}(u,u)\le\varliminf_{n\to+\infty}\mathscr{E}_{\lambda_m}(u_n,u_n)\le\varliminf_{n\to+\infty}\mathscr{E}_{\lambda_n}(u_n,u_n).$$
By Theorem \ref{thm_incre} again, we have
$$\Ee_\loc(u,u)=\lim_{m\to+\infty}\mathscr{E}_{\lambda_m}(u,u)\le\varliminf_{n\to+\infty}\mathscr{E}_{\lambda_n}(u_n,u_n).$$
Hence $\mathscr{E}_{\lambda_n}$ converges to $\Ee_\loc$ in the sense of Mosco.
\end{proof}
Mosco convergence in Theorem \ref{thm_conv_main} implies that appropriate time-changed jump processes can approximate the diffusion at least in the sense of finite-dimensional distribution.

\section{Proof of Theorem \ref{thm_trace_main}}

Similar to Lemma \ref{lem_equiv1}, we have the following result for the unit interval. For all $u\in L^2(I)$, we have
$$\int_{I}\int_I\frac{(u(x)-u(y))^2}{|x-y|^{1+\beta}}\md x\md y\asymp\sum_{n=0}^\infty2^{n}2^{\beta n}\int_I\int_{B(x,2^{-n})}(u(x)-u(y))^2\md y\md x.$$
Combining this result with Equation (\ref{eqn_interval}), we obtain that for all $u\in C(I)$
$$\sum_{n=1}^\infty2^{-n}2^{\beta n}\sum_{i=0}^{2^n-1}(u(\frac{i}{2^n})-u(\frac{i+1}{2^n}))^2\asymp\sum_{n=0}^\infty2^{n}2^{\beta n}\int_I\int_{B(x,2^{-n})}(u(x)-u(y))^2\md y\md x.$$
Since $\beta_1\in(\alpha_1,\beta_1^*)$, we have $\beta_2=\beta_1-(\alpha_1-\alpha_2)\in(\alpha_2,\beta_1^*-\alpha_1+\alpha_2)\subseteq(\alpha_2,\beta_2^*)$. For all $u\in B^{2,2}_{\alpha_1,\beta_1}(K)$, we have $u\in C(K)$, hence $u|_I\in C(I)$. Note that
$$\sum_{n=1}^\infty2^{-\alpha_2n}2^{\beta_2 n}\sum_{i=0}^{2^n-1}(u(\frac{i}{2^n})-u(\frac{i+1}{2^n}))^2\le\sum_{n=1}^\infty2^{-\alpha_1 n}2^{\beta_1 n}\sum_{w\in W_n}\sum_{p,q\in V_w}(u(p)-u(q))^2<+\infty,$$
hence $u|_I\in B^{2,2}_{\alpha_2,\beta_2}(I)$.$\square$

\section{Proof of Theorem \ref{thm_jumping_kernel}}

First, we construct equivalent semi-norms with jumping kernels that converge exactly to local Dirichlet form.

For all $\lambda\in(1/5,1/3)$, $(\mathscr{E}_{\lambda},\Ff_{-\log\lambda/\log2})$ is a non-local regular Dirichlet form on $L^2(K;\nu)$, by Beurling-Deny formula, there exists a unique jumping measure $J_\lambda$ on $K\times K\backslash\mathrm{diag}$ such that for all $u\in\Ff_{-\log\lambda/\log2}$, we have
$$\mathscr{E}_{\lambda}(u,u)=\iint_{K\times K\backslash\mathrm{diag}}(u(x)-u(y))^2J_\lambda(\md x\md y).$$
It is obvious that
$$J_{\lambda}(\md x\md y)=(5\lambda-1)\sum_{n=1}^\infty\frac{1}{(3\lambda)^n}\sum_{w\in W_n}\sum_{p,q\in V_w}\delta_p(\md x)\delta_q(\md y),$$
where $\delta_p,\delta_q$ are Dirac measures at $p,q$, respectively. Hence $J_\lambda$ is singular with respect to $\nu\times\nu$ and no jumping kernel exists. Since
$$
\begin{aligned}
&\sum_{n=1}^\infty\frac{1}{(3\lambda)^n}\sum_{w\in W_n}\sum_{p,q\in V_w}(u(p)-u(q))^2=\sum_{n=1}^\infty2^{(\beta-\alpha)n}\sum_{w\in W_n}\sum_{p,q\in V_w}(u(p)-u(q))^2\\
&=\iint_{K\times K\backslash\mathrm{diag}}(u(x)-u(y))^2\left(\sum_{n=1}^\infty2^{(\beta-\alpha)n}\sum_{w\in W_n}\sum_{p,q\in V_w}\delta_p(\md x)\delta_q(\md y)\right),
\end{aligned}
$$
where
$$
\begin{aligned}
&\sum_{n=1}^\infty2^{(\beta-\alpha)n}\sum_{w\in W_n}\sum_{p,q\in V_w}\delta_p(\md x)\delta_q(\md y)=\sum_{n=1}^\infty\sum_{w\in W_n}\sum_{p,q\in V_w}\frac{1}{|p-q|^{\alpha+\beta}}{|p-q|^{2\alpha}}\delta_p(\md x)\delta_q(\md y)\\
&=\frac{1}{|x-y|^{\alpha+\beta}}\sum_{n=1}^\infty2^{-2\alpha n}\sum_{w\in W_n}\sum_{p,q\in V_w}\delta_p(\md x)\delta_q(\md y)=\frac{1}{|x-y|^{\alpha+\beta}}J(\md x\md y),
\end{aligned}
$$
and
$$J(\md x\md y)=\sum_{n=1}^\infty2^{-2\alpha n}\sum_{w\in W_n}\sum_{p,q\in V_w}\delta_p(\md x)\delta_q(\md y).$$

\begin{myprop}\label{prop_kernel1}
Let
$$c_i(x,y)=\sum_{n=1}^\infty2^{-2\alpha n}\sum_{w\in W_n}\sum_{p,q\in V_w}\frac{1}{\nu(K^{(i)}_{p,n})\nu(K^{(i)}_{q,n})}1_{K^{(i)}_{p,n}}(x)1_{K^{(i)}_{q,n}}(y),$$
then for all $u\in C(K)$, we have
\begin{equation}\label{eqn_kernel1}
\begin{aligned}
&\left(1-C(\frac{2^{\alpha-\gamma i}}{1-2^{\alpha-\gamma i}}+\frac{2^{\alpha-\frac{\beta-\alpha}{2}\gamma i}}{1-2^{\alpha-\frac{\beta-\alpha}{2}\gamma i}})\right)\sum_{n=1}^\infty 2^{(\beta-\alpha)n}\sum_{w\in W_n}\sum_{p,q\in V_w}(u(p)-u(q))^2\\
&\le\iint_{K\times K\backslash\mathrm{diag}}\frac{c_i(x,y)(u(x)-u(y))^2}{|x-y|^{\alpha+\beta}}\nu(\md x)\nu(\md y)\\
&\le\left(1+C(\frac{2^{\alpha-\gamma i}}{1-2^{\alpha-\gamma i}}+\frac{2^{\alpha-\frac{\beta-\alpha}{2}\gamma i}}{1-2^{\alpha-\frac{\beta-\alpha}{2}\gamma i}})\right)\sum_{n=1}^\infty 2^{(\beta-\alpha)n}\sum_{w\in W_n}\sum_{p,q\in V_w}(u(p)-u(q))^2.
\end{aligned}
\end{equation}
\end{myprop}

\begin{proof}
Note that
$$
\begin{aligned}
&\iint_{K\times K\backslash\mathrm{diag}}\frac{c_i(x,y)(u(x)-u(y))^2}{|x-y|^{\alpha+\beta}}\nu(\md x)\nu(\md y)\\
&=\sum_{n=1}^\infty2^{-2\alpha n}\sum_{w\in W_n}\sum_{p,q\in V_w}\frac{1}{\nu(K^{(i)}_{p,n})\nu(K^{(i)}_{q,n})}\int_{K^{(i)}_{p,n}}\int_{K^{(i)}_{q,n}}\frac{(u(x)-u(y))^2}{|x-y|^{\alpha+\beta}}\nu(\md x)\nu(\md y).
\end{aligned}
$$
Since
$$\frac{1}{\nu(K^{(i)}_{p,n})\nu(K^{(i)}_{q,n})}1_{K^{(i)}_{p,n}}(x)1_{K^{(i)}_{q,n}}(y)\nu(\md x)\nu(\md y)\text{ converges weakly to }\delta_p(\md x)\delta_q(\md y),$$
for all $u\in C(K)$, we have
$$
\begin{aligned}
&\lim_{i\to+\infty}\frac{1}{\nu(K^{(i)}_{p,n})\nu(K^{(i)}_{q,n})}\int_{K^{(i)}_{p,n}}\int_{K^{(i)}_{q,n}}\frac{(u(x)-u(y))^2}{|x-y|^{\alpha+\beta}}\nu(\md x)\nu(\md y)\\
&=\frac{(u(p)-u(q))^2}{|p-q|^{\alpha+\beta}}=2^{(\alpha+\beta)n}(u(p)-u(q))^2.
\end{aligned}
$$
By Fatou's lemma, we have
$$
\begin{aligned}
&\sum_{n=1}^\infty 2^{(\beta-\alpha)n}\sum_{w\in W_n}\sum_{p,q\in V_w}(u(p)-u(q))^2\\
&\le\varliminf_{i\to+\infty}\iint_{K\times K\backslash\mathrm{diag}}\frac{c_i(x,y)(u(x)-u(y))^2}{|x-y|^{\alpha+\beta}}\nu(\md x)\nu(\md y).
\end{aligned}
$$
If $\text{LHS}=+\infty$, then $E(u)=+\infty$, the limit in RHS exists and equals to $+\infty$. Hence, we may assume that $E(u)<+\infty$, by Lemma \ref{lem_holder}, we have
$$|u(x)-u(y)|^2\le cE(u)|x-y|^{\beta-\alpha}\text{ for all }x,y\in K.$$
Consider
$$
\begin{aligned}
&\lvert\frac{1}{\nu(K^{(i)}_{p,n})\nu(K^{(i)}_{q,n})}\int_{K^{(i)}_{p,n}}\int_{K^{(i)}_{q,n}}\frac{(u(x)-u(y))^2}{|x-y|^{\alpha+\beta}}\nu(\md x)\nu(\md y)-\frac{(u(p)-u(q))^2}{|p-q|^{\alpha+\beta}}\rvert\\
&\le\frac{1}{\nu(K^{(i)}_{p,n})\nu(K^{(i)}_{q,n})}\int_{K^{(i)}_{p,n}}\int_{K^{(i)}_{q,n}}\lvert\frac{(u(x)-u(y))^2}{|x-y|^{\alpha+\beta}}-\frac{(u(p)-u(q))^2}{|p-q|^{\alpha+\beta}}\rvert\nu(\md x)\nu(\md y).
\end{aligned}
$$
For all $x\in K^{(i)}_{p,n}, y\in K^{(i)}_{q,n}$, we have
$$
\begin{aligned}
&\lvert\frac{(u(x)-u(y))^2}{|x-y|^{\alpha+\beta}}-\frac{(u(p)-u(q))^2}{|p-q|^{\alpha+\beta}}\rvert\\
&\le\frac{1}{|x-y|^{\alpha+\beta}|p-q|^{\alpha+\beta}}\left((u(x)-u(y))^2\lvert |p-q|^{\alpha+\beta}-|x-y|^{\alpha+\beta}\rvert\right.\\
&\left.+\lvert(u(x)-u(y))^2-(u(p)-u(q))^2\rvert\cdot |x-y|^{\alpha+\beta}\right).
\end{aligned}
$$
Since
$$|p-q|\ge|x-y|\ge|p-q|-|x-p|-|y-q|\ge|p-q|-\frac{2}{2^{\gamma ni}}|p-q|=\left(1-\frac{2}{2^{\gamma ni}}\right)|p-q|,$$
for all $i\ge2$, we have $|x-y|\ge|p-q|/2$, for all $i\ge1$, we have
$$|p-q|^{\alpha+\beta}\ge|x-y|^{\alpha+\beta}\ge\left(1-\frac{2}{2^{\gamma ni}}\right)^{\alpha+\beta}|p-q|^{\alpha+\beta}.$$
Hence
$$\frac{1}{|x-y|^{\alpha+\beta}|p-q|^{\alpha+\beta}}\le\frac{2^{\alpha+\beta}}{|p-q|^{2(\alpha+\beta)}}=2^{\alpha+\beta}2^{2(\alpha+\beta)n},$$
$$(u(x)-u(y))^2\le cE(u)|x-y|^{\beta-\alpha}\le cE(u)|p-q|^{\beta-\alpha}=cE(u)2^{-(\beta-\alpha)n},$$
$$\lvert |p-q|^{\alpha+\beta}-|x-y|^{\alpha+\beta}\rvert\le|p-q|^{\alpha+\beta}\left[1-(1-\frac{2}{2^{\gamma ni}})^{\alpha+\beta}\right]\le 2(\alpha+\beta)2^{-(\alpha+\beta)n-\gamma ni},$$
$$
\begin{aligned}
&\lvert(u(x)-u(y))^2-(u(p)-u(q))^2\rvert\\
&=\lvert(u(x)-u(y))+(u(p)-u(q))\rvert\cdot\lvert(u(x)-u(y))-(u(p)-u(q))\rvert\\
&\le\left(|u(x)-u(y)|+|u(p)-u(q)|\right)\left(|u(x)-u(p)|+|u(y)-u(q)|\right)\\
&\le cE(u)\left(|x-y|^{\frac{\beta-\alpha}{2}}+|p-q|^{\frac{\beta-\alpha}{2}}\right)\left(|x-p|^{\frac{\beta-\alpha}{2}}+|y-q|^{\frac{\beta-\alpha}{2}}\right)\\
&\le 4cE(u)2^{-\frac{\beta-\alpha}{2}(2n+\gamma ni)},
\end{aligned}
$$
$$|x-y|^{\alpha+\beta}\le|p-q|^{\alpha+\beta}=2^{-(\alpha+\beta)n},$$
hence
$$\lvert\frac{(u(x)-u(y))^2}{|x-y|^{\alpha+\beta}}-\frac{(u(p)-u(q))^2}{|p-q|^{\alpha+\beta}}\rvert\le2^{\alpha+\beta}cE(u)\left(2(\alpha+\beta)2^{2\alpha n-\gamma ni}+4\cdot 2^{2\alpha n-\frac{\beta-\alpha}{2}\gamma ni}\right),$$
hence
$$
\begin{aligned}
&|\sum_{n=1}^\infty 2^{(\beta-\alpha)n}\sum_{w\in W_n}\sum_{p,q\in V_w}(u(p)-u(q))^2\\
&-\sum_{n=1}^\infty2^{-2\alpha n}\sum_{w\in W_n}\sum_{p,q\in V_w}\frac{1}{\nu(K^{(i)}_{p,n})\nu(K^{(i)}_{q,n})}\int_{K^{(i)}_{p,n}}\int_{K^{(i)}_{q,n}}\frac{(u(x)-u(y))^2}{|x-y|^{\alpha+\beta}}\nu(\md x)\nu(\md y)|\\
&\le\sum_{n=1}^\infty 2^{-2\alpha n}2^{\alpha n}2^{\alpha+\beta}cE(u)\left(2(\alpha+\beta)2^{2\alpha n-\gamma ni}+4\cdot 2^{2\alpha n-\frac{\beta-\alpha}{2}\gamma ni}\right)\\
&\le CE(u)\sum_{n=1}^\infty\left(2^{\alpha n-\gamma ni}+2^{\alpha n-\frac{\beta-\alpha}{2}\gamma ni}\right)=CE(u)\sum_{n=1}^\infty\left(2^{(\alpha-\gamma i)n}+2^{(\alpha-\frac{\beta-\alpha}{2}\gamma i)n}\right).
\end{aligned}
$$
Choose $\gamma\ge1$ such that $\alpha-\gamma<0$ and $\alpha-\frac{\beta-\alpha}{2}\gamma<0$, then
$$\sum_{n=1}^\infty\left(2^{(\alpha-\gamma i)n}+2^{(\alpha-\frac{\beta-\alpha}{2}\gamma i)n}\right)=\frac{2^{\alpha-\gamma i}}{1-2^{\alpha-\gamma i}}+\frac{2^{\alpha-\frac{\beta-\alpha}{2}\gamma i}}{1-2^{\alpha-\frac{\beta-\alpha}{2}\gamma i}}\to0,$$
as $i\to+\infty$. Hence
$$
\begin{aligned}
&|\iint_{K\times K\backslash\mathrm{diag}}\frac{c_i(x,y)(u(x)-u(y))^2}{|x-y|^{\alpha+\beta}}\nu(\md x)\nu(\md y)-\sum_{n=1}^\infty 2^{(\beta-\alpha)n}\sum_{w\in W_n}\sum_{p,q\in V_w}(u(p)-u(q))^2|\\
&\le CE(u)\left(\frac{2^{\alpha-\gamma i}}{1-2^{\alpha-\gamma i}}+\frac{2^{\alpha-\frac{\beta-\alpha}{2}\gamma i}}{1-2^{\alpha-\frac{\beta-\alpha}{2}\gamma i}}\right),
\end{aligned}
$$
hence we have Equation (\ref{eqn_kernel1}).
\end{proof}

Second, we do appropriate cutoff to have bounded jumping kernels.

\begin{myprop}\label{prop_kernel2}
For all sequence $\myset{\beta_i}\subseteq(\alpha,\beta^*)$ with $\beta_i\uparrow\beta^*$. Let
$$C_i(x,y)=\sum_{n=1}^{\Phi(i)}2^{-2\alpha n}\sum_{w\in W_n}\sum_{p,q\in V_w}\frac{1}{\nu(K^{(i)}_{p,n})\nu(K^{(i)}_{q,n})}1_{K^{(i)}_{p,n}}(x)1_{K^{(i)}_{q,n}}(y),$$
where $\Phi:\mathbb{N}\to\mathbb{N}$ is increasing and $(5\cdot2^{-\beta_i}-1)\Phi(i)\ge i$ for all $i\ge1$.
Then for all $u\in\Ff_\loc$, we have
$$\lim_{i\to+\infty}(5\cdot 2^{-\beta_i}-1)\iint_{K\times K\backslash\mathrm{diag}}\frac{C_i(x,y)(u(x)-u(y))^2}{|x-y|^{\alpha+\beta_i}}\nu(\md x)\nu(\md y)=\Ee_\loc(u,u).$$
\end{myprop}

\begin{proof}
By the proof of Proposition \ref{prop_kernel1}, for all $u\in\Ff_\loc$, we have
$$
\begin{aligned}
&\lvert\iint_{K\times K\backslash\mathrm{diag}}\frac{C_i(x,y)(u(x)-u(y))^2}{|x-y|^{\alpha+\beta_i}}\nu(\md x)\nu(\md y)-\sum_{n=1}^{\Phi(i)}2^{(\beta_i-\alpha)n}\sum_{w\in W_n}\sum_{p,q\in V_w}(u(p)-u(q))^2\rvert\\
&\le CE_{\beta_i}(u,u)\sum_{n=1}^{\Phi(i)}\left(2^{(\alpha-\gamma i)n+2^{(\alpha-\frac{\beta_i-\alpha}{2}\gamma i)n}}\right)\le CE_{\beta_i}(u,u)\left(\frac{2^{\alpha-\gamma i}}{1-2^{\alpha-\gamma i}}+\frac{2^{\alpha-\frac{\beta_i-\alpha}{2}\gamma i}}{1-2^{\alpha-\frac{\beta_i-\alpha}{2}\gamma i}}\right),
\end{aligned}
$$
hence
$$
\begin{aligned}
&\lvert(5\cdot2^{-\beta_i}-1)\iint_{K\times K\backslash\mathrm{diag}}\frac{C_i(x,y)(u(x)-u(y))^2}{|x-y|^{\alpha+\beta_i}}\nu(\md x)\nu(\md y)\\
&-(5\cdot2^{-\beta_i}-1)\sum_{n=1}^{\Phi(i)}2^{(\beta_i-\alpha)n}\sum_{w\in W_n}\sum_{p,q\in V_w}(u(p)-u(q))^2\rvert\\
&\le C(5\cdot2^{-\beta_i}-1)E_{\beta_i}(u,u)\left(\frac{2^{\alpha-\gamma i}}{1-2^{\alpha-\gamma i}}+\frac{2^{\alpha-\frac{\beta_i-\alpha}{2}\gamma i}}{1-2^{\alpha-\frac{\beta_i-\alpha}{2}\gamma i}}\right)\to0,
\end{aligned}
$$
as $i\to+\infty$. Hence we only need to show that for all $u\in\Ff_\loc$
$$\lim_{i\to+\infty}(5\cdot2^{-\beta_i}-1)\sum_{n=1}^{\Phi(i)}2^{(\beta_i-\alpha)n}\sum_{w\in W_n}\sum_{p,q\in V_w}(u(p)-u(q))^2=\Ee_\loc(u,u).$$
Let $\lambda_i=2^{-\beta_i}$, then $\myset{\lambda_i}\subseteq(1/5,1/3)$ and $\lambda_i\downarrow1/5$. We use the notions of the proof of Theorem \ref{thm_incre}. We only need to show that for all $u\in\Ff_\loc$
$$\lim_{i\to+\infty}(5\lambda_i-1)\sum_{n=1}^{\Phi(i)}\frac{1}{(5\lambda_i)^n}a_n=a_\infty.$$
It is obvious that
$$\varlimsup_{i\to+\infty}(5\lambda_i-1)\sum_{n=1}^{\Phi(i)}\frac{1}{(5\lambda_i)^n}a_n\le\varlimsup_{i\to+\infty}(5\lambda_i-1)\sum_{n=1}^{\infty}\frac{1}{(5\lambda_i)^n}a_n=a_\infty.$$
On the other hand, for all $A<a_\infty$, there exists $N_0\ge1$ such that $a_n>A$ for all $n>N_0$, hence
$$
\begin{aligned}
&(5\lambda_i-1)\sum_{n=1}^{\Phi(i)}\frac{1}{(5\lambda_i)^n}a_n\ge(5\lambda_i-1)\sum_{n=N_0+1}^{\Phi(i)}\frac{1}{(5\lambda_i)^n}A\\
&=(5\lambda_i-1)\frac{\frac{1}{(5\lambda_i)^{N_0+1}}\left(1-\frac{1}{(5\lambda_i)^{\Phi(i)-N_0}}\right)}{1-\frac{1}{5\lambda_i}}A=\frac{1}{(5\lambda_i)^{N_0}}\left(1-\frac{1}{(5\lambda_i)^{\Phi(i)-N_0}}\right)A.\\
\end{aligned}
$$
It is obvious that $1/(5\lambda_i)^{N_0}\to1$ as $i\to+\infty$. Since $(5\lambda_i-1)\Phi(i)\ge i$, we have
$$
\begin{aligned}
(5\lambda_i)^{\Phi(i)}&=(1+5\lambda_i-1)^{\Phi(i)}=\left[(1+5\lambda_i-1)^{\frac{1}{5\lambda_i-1}}\right]^{(5\lambda_i-1)\Phi(i)}\\
&\ge\left[(1+5\lambda_i-1)^{\frac{1}{5\lambda_i-1}}\right]^{i}\to+\infty.
\end{aligned}
$$
Hence
$$\varliminf_{i\to+\infty}(5\lambda_i-1)\sum_{n=1}^{\Phi(i)}\frac{1}{(5\lambda_i)^n}a_n\ge A$$
for all $A<a_\infty$, hence
$$\lim_{i\to+\infty}(5\lambda_i-1)\sum_{n=1}^{\Phi(i)}\frac{1}{(5\lambda_i)^n}a_n=a_\infty.$$
\end{proof}

Now we give the proof of Theorem \ref{thm_jumping_kernel}.

\begin{proof}[Proof of Theorem \ref{thm_jumping_kernel}]
For all $u\in\Ff_\loc$, by Proposition \ref{prop_kernel2},we have
$$\lim_{i\to+\infty}(5\cdot 2^{-\beta_i}-1)\iint_{K\times K\backslash\mathrm{diag}}\frac{C_i(x,y)(u(x)-u(y))^2}{|x-y|^{\alpha+\beta_i}}\nu(\md x)\nu(\md y)=\Ee_\loc(u,u),$$
by Theorem \ref{thm_main} and Theorem \ref{thm_incre}, we have
$$
\begin{aligned}
\frac{1}{C}\Ee_\loc(u,u)&\le\varliminf_{i\to+\infty}(5\cdot 2^{-\beta_i}-1)\iint_{K\times K\backslash\mathrm{diag}}\frac{(u(x)-u(y))^2}{|x-y|^{\alpha+\beta_i}}\nu(\md x)\nu(\md y)\\
&\le\varlimsup_{i\to+\infty}(5\cdot 2^{-\beta_i}-1)\iint_{K\times K\backslash\mathrm{diag}}\frac{(u(x)-u(y))^2}{|x-y|^{\alpha+\beta_i}}\nu(\md x)\nu(\md y)\le C\Ee_\loc(u,u),
\end{aligned}
$$
hence
$$\lim_{i\to+\infty}(1-\delta_i)(5\cdot 2^{-\beta_i}-1)\iint_{K\times K\backslash\mathrm{diag}}\frac{(u(x)-u(y))^2}{|x-y|^{\alpha+\beta_i}}\nu(\md x)\nu(\md y)=0,$$
hence
$$\lim_{i\to+\infty}(5\cdot 2^{-\beta_i}-1)\iint_{K\times K\backslash\mathrm{diag}}\frac{a_i(x,y)(u(x)-u(y))^2}{|x-y|^{\alpha+\beta_i}}\nu(\md x)\nu(\md y)=\Ee_\loc(u,u).$$
It is obvious that $a_i=\delta_iC_i+(1-\delta_i)$ is bounded from above and below by positive constants.
\end{proof}

\def\cprime{$'$}

\end{document}